\documentclass[12pt]{amsart}
\usepackage{mathptmx}
\usepackage{amsmath,amsfonts,amssymb}
\usepackage{mathtools}
\usepackage{mathrsfs}
\usepackage[english]{babel}
\usepackage{epsfig}
\usepackage[all]{xy}
\usepackage{graphicx}
\usepackage{latexsym}
\usepackage{verbatim}
\usepackage{enumerate}
\usepackage{ulem}
\usepackage{cancel}
\usepackage{tabularx}
\usepackage{caption}
\usepackage{hyperref} 
\allowdisplaybreaks
\mathtoolsset{showonlyrefs,showmanualtags}
\usepackage{xcolor}

\newcommand{\D}{\mathbb D}
\newcommand{\RH}{\mathrm{RH}}

\newcommand{\C}{\mathbb C}

\newcommand{\Aut}{{\sf Aut}}

\def\eps{\varepsilon}
\def\Re{{\sf Re}\,}

\newcommand{\cc}[1]{\overline{{#1}}}

\def\dD{\mathop{{\rm d}_\D}}
\def\dRH{\mathop{{\rm d}_\RH}}

\newcommand{\mcite}[1]{\csname b@#1\endcsname}
\newtheoremstyle{break}
{8pt}{8pt}%
{\itshape}{}%
{\bfseries}{}%
{\newline}{}%
\theoremstyle{break}

\theoremstyle{break}
\newtheorem{theorem}{Theorem}[section]
\newtheorem*{theorem*}{Theorem}
\newtheorem{lemma}[theorem]{Lemma}
\newtheorem*{lemma*}{Lemma}

\newtheorem{corollary}[theorem]{Corollary}

\theoremstyle{break}
\newtheorem{theoremliterature}{Theorem}

\theoremstyle{break}
\newtheorem{definition}[theorem]{Definition}

\theoremstyle{remark}
\newtheorem{remark}[theorem]{Remark}
\newtheorem*{concludingremarks}{Concluding remarks}

\numberwithin{equation}{section}

\author[A.~Moucha]{Annika Moucha
	$^\dag$}
\address{A. Moucha: Department of Mathematics, University of W\"urzburg, Emil Fischer Strasse 40, 97074, W\"urzburg, Germany.
} \email{annika.moucha@uni-wuerzburg.de
}

\thanks{$^\dag\,$Partially supported by the Alexander von Humboldt Stiftung
}

\title{A Burns-Krantz type theorem for Blaschke products}

\begin{document}

\begin{abstract}
	Let $f$ be a holomorphic function mapping the open unit disk 
	into itself.  We establish a boundary version of Schwarz' lemma in the spirit of a result by Burns and Krantz and provide sufficient conditions on the local behaviour of~$f$ near some boundary point that forces~$f$ to be a Blaschke product with predescribed critical points. For the proof, a local Julia type inequality based on Nehari’s sharpening of Schwarz' lemma is established.
\end{abstract}

\maketitle

\section{Introduction}

Denote by $\D:=\{z\in\C \,:\, |z|<1\}$ the \emph{open unit disk}. The classical Schwarz lemma states that every holomorphic function $f:\D\to\D$ that fixes the origin satisfies either
\[|f'(0)|<1\quad\text{and}\quad |f(z)|< |z|\quad\text{ for all }z\in\D\setminus\{0\}\]
or $|f'(0)|=1$ and, in this case, $f$ coincides with the rotation $f(z)=f'(0)z$. The second case immediately implies the following \emph{rigidity principle}: every holomorphic function $f:\D\to\D$ that fixes some point $p\in\D$ and satisfies $f'(p)=1$ coincides with the identity function, i.e.\ $f(z)=z$. In other words: every holomorphic function $f:\D\to\D$ which locally --- that is at $p\in\D$ --- agrees with the identity function up to first order already is the identity function. In the present paper we are interested in boundary versions of this rigidity principle in the following sense.
\begin{theoremliterature}[Burns-Krantz (1994); see~{\cite[Th.~2.1]{burnsRigidityHolomorphicMappings1994}}]\label{th:BurnsKrantz}
	Let $f:\D\to\D$ be a holomorphic function and $\xi\in\partial\D$. If
	\begin{equation}\label{eq:BKcondition}
		f(z)=z+o\big(|\xi-z|^3\big),\quad \text{as }z\to\xi,
	\end{equation}
	then $f(z)=z$ for all $z\in\D$.
\end{theoremliterature}
Theorem~\ref{th:BurnsKrantz} pioneered a multitude of ``boundary Schwarz lemmas'', by which we mean rigidity principles involving one or several boundary points. The survey~\cite[Sec.~5]{elinSchwarzLemmaRigidity2014} provides a detailed list of extensions and variations of Theorem~\ref{th:BurnsKrantz} (up to 2014); see also\ \cite{bracciStrongFormAhlforsSchwarz2024,dubininSchwarzInequalityBoundary2004,liuSchwarzLemmaBoundary2016,ossermanSharpSchwarzInequality2000,shoikhetAnotherLookBurnsKrantz2008,tangSchwarzLemmaBoundary2017,taurasoRigidityBoundaryHolomorphic2001,zimmerTwoBoundaryRigidity2022} for further references and more recent work.

In order to place the results of this paper into context, we present two particular ``boundary Schwarz lemmas''. The first one states that the assumption on the local behaviour of~$f$ at~$\xi$ in Theorem~\ref{th:BurnsKrantz} can be weakened.
\begin{theoremliterature}[Baracco-Zaitsev-Zampieri (2006); see~{\cite[Prop.~3.2]{baraccoBurnsKrantzTypeTheorem2006}}]\label{th:BaraccoEtAl}
	Let $f:\D\to\D$ be a holomorphic function and $\xi\in\partial\D$. If there is a sequence $(z_n)\subseteq\D$ such that $z_n\to\xi$ non-tangentially as $n\to\infty$ and
	\begin{equation}
		f(z_n)=z_n+o\big(|\xi-z_n|^3\big),\quad \text{as }n\to\infty,
	\end{equation}
	then $f(z)=z$ for all $z\in\D$.
\end{theoremliterature}
A second extension of Theorem~\ref{th:BurnsKrantz} is the following result that gives sufficient conditions on a holomorphic function $f:\D\to\D$ in order to coincide with a finite Blaschke product.\footnote{See also~\cite{bolotnikovUniquenessResultBoundary2008} for a generalization of Theorem~\ref{th:Chelst} by using a different approach than in~\cite{chelstGeneralizedSchwarzLemma2001}.}
\begin{theoremliterature}[Chelst (2001); see~{\cite[Th.~2]{chelstGeneralizedSchwarzLemma2001}}]\label{th:Chelst}
	Let $f:\D\to\D$ be a holomorphic function and $B$ a finite Blaschke product of degree~$n$. Let further $\sigma\in\partial\D$ and denote $B^{-1}(\{\sigma\})=\{\xi_1,\dots,\xi_n\}$. If
	\begin{equation}
		\begin{split}
			f(z)&=B(z)+o\big(|\xi_1-z|^3\big),\quad \text{as }z\to\xi_1 
			;\\
			f(z)&=B(z)+o\big(|\xi_k-z|\big), \; \; \quad \text{as }z\to\xi_k 
			,\,\text{ for } k\geq 2,
		\end{split}
	\end{equation}
	then $f(z)=B(z)$ for all $z\in\D$.
\end{theoremliterature}
The main result of this paper combines the underlying ideas of Theorems~\ref{th:BaraccoEtAl} and~\ref{th:Chelst} --- namely controlling the local behaviour of~$f$ near a given boundary point only on a single (non-tangential) sequence (Theorem~\ref{th:BaraccoEtAl}) and comparing~$f$ to a more general function than the identity function (Theorem~\ref{th:Chelst}) --- by additionally taking critical points into account. More specifically, our approach involves so-called \emph{maximal} Blaschke products, a type of Blaschke product that is intimately tied to the critical points of holomorphic self-maps of~$\D$. We give a precise definition of maximal Blaschke products below (see Definition~\ref{def:MBP}). At this point we only like to point out that every finite Blaschke product is a maximal Blaschke product. However, the class of maximal Blaschke products is much bigger as it contains certain infinite Blaschke products, too (see Remark~\ref{rem:MBP}\eqref{it:FBPsubsetMBP}).
\begin{theorem}\label{thm:sequenceBurnsKrantzforMBP}
	Let $f:\D\to\D$ be a holomorphic function, $\xi\in\partial\D$ and $B$ a maximal Blaschke product for~$f$. Further assume, that
	\begin{equation}\label{eq:MBPhasfiniteBoundaryDilationCoef}
		\liminf_{z\to\xi}\frac{1-|B(z)|}{1-|z|}\in(0,\infty).
	\end{equation}
	If there is a sequence $(z_n)\subseteq\D$ such that $z_n\to\xi$ non-tangentially as $n\to\infty$ and
	\begin{equation}\label{eq:MBPfequalsBnontangonsequence}
		f(z_n)=B(z_n)+o(|\xi-z_n|^3)\quad \text{as }n\rightarrow\infty\,,
	\end{equation}
	then $f(z)=B(z)$ for all $z\in\D$.
\end{theorem}

Choosing~$B$ to be the identity function in Theorem~\ref{thm:sequenceBurnsKrantzforMBP} recovers Theorem~\ref{th:BaraccoEtAl} --- and hence also implies Theorem~\ref{th:BurnsKrantz}. Further note that the assumption~\eqref{eq:MBPhasfiniteBoundaryDilationCoef} demands~$B$ to have a finite angular derivative (see Section~\ref{sub:AngularDerivative}) at~$\xi$. Since this property is always satisfied for finite Blaschke products which are, as it was mentioned above, maximal Blaschke products, Theorem~\ref{thm:sequenceBurnsKrantzforMBP} has the following immediate consequence.
\begin{corollary}\label{cor:sequenceBurnsKrantzforMBP}
	Let $B$ be a finite Blaschke product and $f:\D\to\D$ be a holomorphic function such that $B$ is a (finite) maximal Blaschke product for~$f$. Further, let $\xi\in\partial\D$. If there is a sequence $(z_n)\subseteq\D$ such that $z_n\to\xi$ non-tangentially as $n\to\infty$ and
	\begin{equation}\label{eq:MBPfequalsBnontangonsequenceCorollary}
		f(z_n)=B(z_n)+o(|\xi-z_n|^3)\quad \text{as }n\rightarrow\infty\,,
	\end{equation}
	then $f(z)=B(z)$ for all $z\in\D$.
\end{corollary}
We will see in Section~\ref{sec:MBP} that~$B$ being a MBP for~$f$ implies that every critical point of~$B$ is a critical point of~$f$, too. Therefore, since a finite Blaschke product possesses degree~$B$ minus one many critical points, Corollary~\ref{cor:sequenceBurnsKrantzforMBP} puts degree of~$B$ many constraints on~$f$ w.r.t.~$B$. The same number of conditions relates~$f$ to~$B$ in Theorem~\ref{th:Chelst}. This way, one can view Corollary~\ref{cor:sequenceBurnsKrantzforMBP} as an analogue to Theorem~\ref{th:Chelst}.

The major work in proving Theorem~\ref{thm:sequenceBurnsKrantzforMBP} consists in establishing the conditions of the following recently proven ``boundary Schwarz lemma''.
\begin{theoremliterature}[Bracci-Kraus-Roth (2023), see~{\cite[Th.~2.10]{bracciNewSchwarzPickLemma2023}}]\label{thm:MBPsequenceBurnsKrantzforhypderivatives}
	Let $f:\D\to\D$ be a holomorphic function, $\xi\in\partial\D$ and $B$ a maximal Blaschke product for~$f$. If there is a sequence $(z_n)\subseteq\D$ such that $z_n\to\xi$ non-tangentially and
	\begin{equation}\label{eq:HypDerivcondition}
		\frac{|f'(z_n)|}{|B'(z_n)|}\frac{1-|B(z_n)|^2}{1-|f(z_n)|^2}=1+o\big(|\xi-z_n|^2\big)\quad \text{as }n\to\infty,
	\end{equation}
	then $f=T\circ B$ for some $T\in\Aut(\D)$, i.e.\ some conformal automorphism $T:\D\to\D$.
\end{theoremliterature}

In comparison with the results discussed above, Theorem~\ref{thm:MBPsequenceBurnsKrantzforhypderivatives} can be understood as a different type of ``boundary Schwarz lemma'' in the following sense: instead of a local condition on the behaviour of~$f$, Theorem~\ref{thm:MBPsequenceBurnsKrantzforhypderivatives} imposes a local condition on the behaviour of the so-called hyperbolic distortion $|f'|/(1-|f|^2)$ of~$f$. This way, Theorem~\ref{th:BurnsKrantz} is not directly contained in Theorem~\ref{thm:MBPsequenceBurnsKrantzforhypderivatives} but instead can be obtained in a two-tier fashion: First, one can show that the assumption~\eqref{eq:BKcondition} of Theorem~\ref{th:BurnsKrantz} implies~\eqref{eq:HypDerivcondition} for the function $B(z)=z$. Then, using~\eqref{eq:BKcondition} again combined with Theorem~\ref{thm:MBPsequenceBurnsKrantzforhypderivatives} establishes Theorem~\ref{th:BurnsKrantz} (see \cite[Prop.~8.1]{bracciNewSchwarzPickLemma2023} or \cite[Th.~2.7.4]{abateHolomorphicDynamicsHyperbolic2022} for more details).

In~\cite[Prob.~5.1]{bracciStrongFormAhlforsSchwarz2024} the following question was posed: ``\emph{Does this strengthened version of the Burns–Krantz theorem [Theorem~\ref{th:BaraccoEtAl}] also follow from the boundary Ahlfors–Schwarz lemma for the unit disk [Theorem~\ref{thm:MBPsequenceBurnsKrantzforhypderivatives}]?}'' Since Theorem~\ref{th:BaraccoEtAl} is a special case of Theorem~\ref{thm:sequenceBurnsKrantzforMBP} and our proof utilizes Theorem~\ref{thm:MBPsequenceBurnsKrantzforhypderivatives}, the present work gives, in particular, an affirmative answer to that question.

\smallskip

This paper is organized as follows: First, in Section~\ref{sec:MBP} we discuss maximal Blaschke products. In Section~\ref{sec:GFT} we collect some prerequisites from geometric function theory: Section~\ref{sub:HypGeometry} introduces basic notions about hyperbolic geometry on~$\D$, Section~\ref{sub:AngularDerivative} deals with angular derivatives of holomorphic self-maps of~$\D$, and in Section~\ref{sub:BeardonMinda} we determine certain sets in~$\D$ where existence of a finite angular derivative guarantees injectivity. The ideas in Section~\ref{sub:BeardonMinda} are based on recent work by Beardon and Minda~\cite{beardonGeometricJuliaWolff2023}. Next, in Section~\ref{sec:JuliaInequality} we prove a local Julia type inequality (Lemma~\ref{lem:MBPBlaschkeJulia}). This inequality is one of the crucial ingredients used in the proof of Theorem~\ref{thm:sequenceBurnsKrantzforMBP} because it allows us to use the knowledge of the relative behaviour of the functions~$f$ and~$B$ given on one sequence, i.e.~\eqref{eq:MBPfequalsBnontangonsequence}, in order to obtain information about their relation on a comparably bigger set of points. Finally, in Section~\ref{sec:proof:thm:sequenceBurnsKrantzforMBP} we give the proof of Theorem~\ref{thm:sequenceBurnsKrantzforMBP}.

\section{Maximal Blaschke products (MBP)}\label{sec:MBP}

Recall that $z\in\D$ is a critical point (of multiplicity $m$) of a holomorphic function $f:\D\to\D$ if and only if~$z$ is a zero (of multiplicity~$m$) of~$f'$. Moreover, we denote the collection of critical points of~$f$ counting multiplicities by~$\mathcal{C}_f$. We are interested in the following sharpening of the Schwarz --- or more precisely of the Schwarz-Pick --- lemma.
\begin{theoremliterature}[Kraus (2013); see~{\cite[Cor.~1.5]{krausCriticalSetsBounded2013}}. Kraus-Roth (2013); see~{\cite[Th.~1.1]{krausMaximalBlaschkeProducts2013}}]\label{th:NehariSchwarz}
		Let $f:\D\to\D$ be a holomorphic function and $\mathcal{C}$ a subcollection of~$\mathcal{C}_f$. Then there exists a Blaschke product~$B$ such that $\mathcal{C}_B=\mathcal{C}$ and
		\begin{equation}\label{eq:NehariSchwarz}
			\frac{|f'(z)|}{1-|f(z)|^2} \leq \frac{|B'(z)|}{1-|B(z)|^2}
			\qquad \text{ for all }z\in\D
		\end{equation}
		with equality for one --- and hence every --- $z\in\D\setminus\mathcal{C}$ if and only if $f=T\circ B$ for some ${T\in\Aut(\D)}$.
\end{theoremliterature}
\begin{definition}[Maximal Blaschke products]\label{def:MBP}
	Let $f:\D\to\D$ be a holomorphic function. A Blaschke product~$B$ is called a \emph{maximal Blaschke product (MBP) for~$f$} if~$\mathcal{C}_B$ is a subcollection of~$\mathcal{C}_f$ and~$B$ satisfies
	\begin{equation}
		\frac{|g'(z)|}{1-|g(z)|^2} \leq \frac{|B'(z)|}{1-|B(z)|^2}
		\qquad \text{ for all }z\in\D
	\end{equation}
	for every holomorphic function $g:\D\to\D$ such that~$\mathcal{C}_B$ is also a subcollection of~$\mathcal{C}_g$.
\end{definition}
\begin{remark}\label{rem:MBP}
	\begin{enumerate}[(a)]
		\item Every MBP~$B$ is \emph{indestructible}, i.e.\ if $T\in\Aut(\D)$, then $T\circ B$ is a MBP, too.
		\item Every MBP~$B$ is uniquely determined by~$\mathcal{C}_B$ up to postcomposing with some $T\in\Aut(\D)$.
		\item If $\mathcal{C}=\emptyset$, then Theorem~\ref{th:NehariSchwarz} recovers the (infinitesimal version of the) classical Schwarz-Pick lemma.
		\item The case that~$\mathcal{C}$ is a finite set in Theorem~\ref{th:NehariSchwarz} has first been proven by Nehari in 1947; see~\cite{nehariGeneralizationSchwarzLemma1947}.
		\item \label{it:FBPsubsetMBP} Every finite Blaschke product is a MBP. In fact, a MBP is a finite Blaschke product if and only if~$\mathcal{C}_B$ is finite; see~\cite[Rem.~1.2(b)]{krausCriticalSetsBounded2013}.
		\item For more on MBP we refer the reader to~\cite{bracciNewSchwarzPickLemma2023,ivriiPrescribingInnerParts2019,ivriiCriticalStructuresInner2021,ivriiAnalyticMappingsUnit2024,krausCriticalSetsBounded2013,krausCriticalPointsGauss2013,krausMaximalBlaschkeProducts2013}.
	\end{enumerate}
\end{remark}

\section{Hyperbolic geometry, angular derivative and Stolz regions}\label{sec:GFT}

\subsection{Some facts from hyperbolic geometry}\label{sub:HypGeometry}
We denote by
\begin{equation}\label{eq:HyperbolicMetric}
	\dD(z,w):=2\tanh^{-1}\left\vert\frac{z-w}{1-\cc{z}w}\right\vert
\end{equation}
the \emph{hyperbolic distance} between two points $z,w\in\D$. Further, we write $[z,w]_h$ for the \emph{geodesic line segment} (w.r.t.~$\dD$) joining $z,w\in\D$ and $(\xi,\sigma)_h$ for the infinite \emph{geodesic line} (w.r.t.~$\dD$) with ``end points'' $\xi, \sigma\in\partial\D$.  The \emph{hyperbolic length} of a curve $\gamma:I\to\D$ is defined by
\begin{equation}
	\ell_h(\gamma):=\int_{\gamma}\frac{|dt|}{1-|t|^2}.
\end{equation}
Note that (the infinitesimal version of) the classical Schwarz-Pick lemma implies $\ell_h(g\circ \gamma)\leq \ell_h(\gamma)$ for every holomorphic function $g:\D\to\D$ and equality holds for~$\gamma$ non-constant if and only if $g\in\Aut(\D)$. Using the hyperbolic length, the hyperbolic distance between two points $z,w\in\D$ can be expressed by
\begin{equation}\label{eq:HypMetricAsIntegralLength}
	\dD(z,w)=\ell_h([z,w]_h).
\end{equation}
Moreover, every curve $\gamma$ in~$\D$ connecting $z$ and $w$ satisfies $\dD(z,w)\leq\ell_h(\gamma)$. For the proofs and further information about hyperbolic geometry we refer  the reader to e.g.\ the monographs~\cite{abateHolomorphicDynamicsHyperbolic2022,beardonHyperbolicMetricGeometric2007,bracciContinuousSemigroupsHolomorphic2020}.

\subsection{Angular derivative}\label{sub:AngularDerivative}
Let $\xi\in\partial \D$ and $m>0$. We define the \emph{hyperbolic Stolz region} of width $2m$ anchored at $\xi$ to be the set
	\[S(m,\xi):=\{z\in\D\,:\, \dD(z,(\xi,-\xi)_h)<m\}.\]
We call a sequence $(z_n)\subseteq\D$ \emph{converging non-tangentially} to $\xi$, if $z_n\to\xi$ and there is $m>0$ such that $z_n\in S(m,\xi)$ eventually. In the following let $g:\D\to\D$ be a holomorphic function. We say that~$g$ has \emph{non-tangential (or angular) limit} $\sigma\in\cc{\D}$ at~$\xi$, if
\[\sigma=\lim_{\substack{S(m,\xi) \ni z\to\xi\\ m>0}}g(z)=:\angle\lim_{z\to\xi}g(z).\]
In this case, we write $g(\xi)=\sigma$. Further, we introduce the \emph{boundary dilation coefficient}~$\alpha_g(\xi)$ of~$g$ at~$\xi$ defined by
\[\alpha_g(\xi):=\liminf_{z\to\xi}\frac{1-|g(z)|}{1-|z|}.\]
We use the short-hand notation $\alpha_g:=\alpha_g(1)$ for the boundary dilation coefficient at~1. By the Julia-Wolff-Carath\'{e}odory theorem (see e.g.~\cite[Ch.~4]{shapiroCompositionOperatorsClassical1993}), $\alpha_g(\xi)\in(0,\infty)$ guarantees the existence of a finite non-zero \emph{angular derivative} of~$g$ at~$\xi$ (and vice versa), that is
\[g'(\xi):=\angle\lim_{z\to\xi}g'(z)=\angle\lim_{z\to\xi}\frac{g(\xi)-g(z)}{\xi-z}\in\C\setminus\{0\}.\]
In this case, the angular limit~$g(\xi)\in\partial\D$ always exists and $g'(\xi)=\alpha_g(\xi)g(\xi)\cc{\xi}$. Further note that if $\alpha_g(\xi)\in(0,\infty)$, then the limit inferior in the definition can be replaced by the angular limit, i.e.\ the limit taken along any non-tangential sequence converging to~$\xi$ (see e.g.~\cite[Prop.~1.7.4]{bracciContinuousSemigroupsHolomorphic2020}). For the proof of Theorem~\ref{thm:sequenceBurnsKrantzforMBP} we need the following elementary result, and we include its proof for the sake of completeness.
\begin{lemma}\label{lem:MBPsequenceangularderivative}
	Let $f,g:\D\to\D$ be holomorphic functions and $\xi\in\partial\D$ such that $\alpha_g(\xi)\in(0,\infty)$. If there is a sequence $(z_n)\subseteq\D$ such that $z_n\to\xi$ non-tangentially as $n\to\infty$ and
	\begin{equation}\label{eq:MBPsequencefequalsgangularderiv}
		f(z_n)=g(z_n)+O(|\xi-z_n|)\quad \text{as }n\rightarrow\infty\,,
	\end{equation}
	then $\alpha_f(\xi)\in(0,\infty)$ and the angular limits $f(\xi)$ and $g(\xi)$ coincide. In particular, $f$ has a finite angular derivative $f'(\xi)$ at $\xi$. If
	\begin{equation}\label{eq:MBPsequencefequalsgangularderiv2}
		f(z_n)=g(z_n)+o(|\xi-z_n|)\quad \text{as }n\rightarrow\infty\,,
	\end{equation}
	then the angular derivatives $f'(\xi)$ and $g'(\xi)$ coincide.
\end{lemma}
\begin{proof}
	The assumptions  $(z_n)$ being a non-tangential sequence and \eqref{eq:MBPsequencefequalsgangularderiv} guarantee
	\begin{equation}\label{eq:MBPsequenceO1}
		\frac{1-|f(z_n)|}{1-|z_n|}=\frac{1-|g(z_n)|+O(|1-z_n|)}{1-|z_n|}=\frac{1-|g(z_n)|}{1-|z_n|}+O(1)
	\end{equation}
	as $n\to\infty$. Since  $\alpha_g(\xi)\in(0,\infty)$ it follows $\alpha_f(\xi)\in(0,\infty)$. This proves the existence of the angular derivative $f'(\xi)$ --- and hence the angular limit $f(\xi)$ --- of~$f$ at~$\xi$. Moreover, \eqref{eq:MBPsequencefequalsgangularderiv} ensures $f(\xi)=g(\xi)$.

	If we assume \eqref{eq:MBPsequencefequalsgangularderiv2}, then $O(1)$ can be replaced by $o(1)$ in \eqref{eq:MBPsequenceO1}. This shows $\alpha_f(\xi)=\alpha_g(\xi)$. Hence,
	\begin{equation}
		f'(\xi)=\alpha_f(\xi)f(\xi)\cc{\xi}=\alpha_g(\xi)g(\xi)\cc{\xi}=g'(\xi).\qedhere
	\end{equation}
\end{proof}

\subsection{Injectivity on (ends of) Stolz regions}\label{sub:BeardonMinda}
For the proof of Theorem~\ref{thm:sequenceBurnsKrantzforMBP} (see Section~\ref{sec:proof:thm:sequenceBurnsKrantzforMBP} below) we will exploit the fact that a holomorphic self-map of the open unit disk with finite angular derivative at some boundary point is injective near that boundary point in a non-tangential sense. In order to make this precise, we introduce the following object.
\begin{definition}[End of Stolz region]
 	Let $\xi\in\partial \D$ and $m>0$. For $M>0$ we define the \emph{$M$-th end of $S(m,\xi)$} to be the set
 	\[E(m,\xi,M):=S(m,\xi)\cap H(\xi,M)\]
 	where $H(\xi,M)$ is the horocycle at $\xi$ of radius $1/M$, i.e.\ the set $\{z\in\D\,:\, M|\xi-z|^2<1-|z|^2\}$.
 \end{definition}
Note that every end of a Stolz region $E(m,\xi,M)$ is the intersection of two hyperbolically convex sets and therefore also hyperbolically convex (meaning that for any two points $z,w\in E(m,\xi,M)$ also $[z,w]_h\subseteq E(m,\xi,M)$).

The next lemma is a collection of results obtained in the recent work~\cite{beardonGeometricJuliaWolff2023} of Beardon and Minda (note that they work in the half-plane setting). Since some parts of the statement are slight modifications or only contained in the proofs of~\cite[Sec.~9-10]{beardonGeometricJuliaWolff2023}, we include the general ideas of how to establish Lemma~\ref{lem:Stolzendinclusions}.
 \begin{lemma}\label{lem:Stolzendinclusions}
 	Let $g:\D\to\D$ be a holomorphic function and $\xi\in\partial\D$ such that $\alpha_g(\xi)\in(0,\infty)$ and $g(\xi)=\sigma$. Then for every Stolz region $S(m,\xi)$ there exists $M>0$ such that $g$ is injective on $E(m,\xi,M)$. Moreover:
 	\begin{enumerate}[(i)]
 		\item \label{it:Stolzendinclusions-EpsAndM} For every $\eps>0$ the constant $M$ can be chosen such that
 		\begin{equation}\label{eq:StolzinImageinStolz}
 			E\left(m-\eps,\sigma,\frac{e^\eps M}{\alpha_g(\xi)}\right)\subseteq g\big(E(m,\xi,M)\big)\subseteq E\left(m+\eps,\sigma,\frac{M}{\alpha_g(\xi)}\right).
 		\end{equation}
 		In particular, if $\eps$ is fixed, then~\eqref{eq:StolzinImageinStolz} also holds for $m$ and $M$ replaced by $m'$ and $M'$ such that $m'\leq m$ and $M'\geq M$.

 		\item \label{it:Stolzendinclusions-InclusionChain}The constant $M$ can be chosen such that
 		\begin{multline}
 			g\big(E(m/2,\xi,4M)\big)
 			\subseteq E\left(\frac{5m}{8},\sigma,\frac{4M}{\alpha_g(\xi)}\right)
 			\subseteq
 			g\big(E(3m/4,\xi,2M)\big)\\
 			\subseteq E\left(\frac{7m}{8},\sigma,\frac{2M}{\alpha_g(\xi)}\right)
 			\subseteq
 			g\big(E\left(m,\xi,M\right)\big).
 		\end{multline}
 	\end{enumerate}
 \end{lemma}
 \begin{proof}
 	Without loss of generality we assume $\xi=\sigma=1$. We switch to the right half-plane $\RH:=\{w\in\C \, : \, \Re w > 0\}$: Denote $C:\D\to\RH$, $z\mapsto (z+1)/(1-z)$. We set $F:=C\circ g\circ C^{-1}:\RH\to\RH$ and $\dRH(w,u):=\dD(C^{-1}(w),C^{-1}(u))$ for $w,u\in\RH$. Then $\alpha_g\in(0,\infty)$ implies
 	\[W:=\inf_{w\in\RH}\frac{\Re F(w)}{\Re w}=\frac{1}{\alpha_g}\in(0,\infty).\]
 	Now we apply \cite[Cor.~2]{beardonGeometricJuliaWolff2023} which states that there is $M>0$ such that $F$ is injective on
 	\[C\big(E(m,1,M)\big)=\{w\in\RH\,:\, \dRH\big(w,(0,\infty)\big)<m\}\cap\{w\in\RH\,:\, \Re w>M\}.\]
 	Here, $(0,\infty)$ denotes the positive real axis which is the hyperbolic geodesic line with ``end points'' $0,\infty\in\partial\RH$. Switching back to $\D$ shows that $g$ is injective on $E(m,1,M)$.

 	Part~\eqref{it:Stolzendinclusions-EpsAndM} also follows from the corresponding statement on $\RH$ which is~\cite[Lem.~5]{beardonGeometricJuliaWolff2023}. Note that the statement of~\cite[Lem.~5]{beardonGeometricJuliaWolff2023} only contains the inclusions~\eqref{eq:StolzinImageinStolz} for fixed $\eps>0$. However, the additional claims that~$g$ is injective on~$E(m,\xi,M)$ and that $m$ and $M$ can be replaced by smaller resp.\ larger constants $m'$ and $M'$ are already implicitly contained: the proof of~\cite[Lem.~5]{beardonGeometricJuliaWolff2023} determines the constant $M>0$ such that
 	\begin{align}
 		\text{(a) }& g \text{ is injective on } E(m,1,M);\\
 		\text{(b) }& \dRH(g(z),Wz)<\eps \text{ for all } z\in E(m,1,M);\\
 		\text{(c) }& \dRH(g(z),Wz)<\eps \text{ for all } z\in \partial E(m,1,M)\cap\partial S(m,1).
 	\end{align}
 	Property~(a) shows our injectivity claim. Further, since $E(m',1,M')\subseteq E(m,1,M)$ if $m'\leq m$ and $M'\geq M$, the properties (a), (b) and (c) also hold with $m$ replaced by $m'$ and $M$ replaced by $M'$. This shows our second additional claim.

 	In order to obtain Part~\eqref{it:Stolzendinclusions-InclusionChain} we can apply Part~\eqref{it:Stolzendinclusions-EpsAndM} thrice for $\eps<\min\{m/8,\log2\}$; this follows an idea in~\cite[Proof of Th.~8]{beardonGeometricJuliaWolff2023}.
 \end{proof}

\begin{corollary}\label{cor:MBPregionwhereBisinjective}
	Let $g:\D\to\D$ be a holomorphic function and $\xi\in\partial\D$ such that $\alpha_g(\xi)\in(0,\infty)$ and $g(\xi)=\sigma$. Then there is a simply connected domain $V\subseteq \D$ with $\xi\in\partial V $ such that $g$ is injective on $V$ and such that~$g(V)$ is hyperbolically convex.

	In fact, for every $m>0$ there exists $M>0$ such that we can choose $V$ with $g(V)=E(m,\sigma,M)$ and \begin{multline}\label{eq:InclusionsForV}
		E(4m/7,\xi,2M\alpha_g(\xi))
		\subseteq
		g^{-1}\big(E\left(5m/7,\sigma,2M\right)\big)\cap V
		\subseteq
		E(6m/7,\xi,M\alpha_g(\xi))\\
		\subseteq V
		 \subseteq E(8m/7,\xi,M\alpha_g(\xi)/2).
	\end{multline}
	In particular, in this case, $g$ is injective on~$\cc{V}\cap\D$.
\end{corollary}
\begin{proof}
	Choose $V:=g^{-1}(E\left(m,\sigma,2\tilde{M}/\alpha_g(\xi)\right))\cap E(8m/7,\xi,\tilde{M})$ where $\tilde{M}$ is the constant obtained from Lemma~\ref{lem:Stolzendinclusions}\eqref{it:Stolzendinclusions-InclusionChain} applied to $\tilde{m}=8m/7$.
\end{proof}

\section{A local Julia type inequality}\label{sec:JuliaInequality}

One key ingredient that we need for the proof of Theorem~\ref{thm:sequenceBurnsKrantzforMBP} is the following local Julia type inequality comparing $f$ and $B$.
\begin{lemma}\label{lem:MBPBlaschkeJulia}
	Let $f:\D\to\D$ be a holomorphic function and $B$ a maximal Blaschke product for~$f$. Further assume, that $\alpha_B\in(0,\infty)$ and $B(1)=1$. If there is a sequence $(z_n)\subseteq\D$ such that $z_n\to1$ non-tangentially as $n\to\infty$ and
	\begin{equation}\label{eq:MBPfequalsBnontangonsequenceintegratedNehari}
		f(z_n)=B(z_n)+O(|1-z_n|)\quad \text{as }n\rightarrow\infty\,,
	\end{equation}
	then
	\begin{equation}\label{eq:JuliaCoefficient}
		A:=\lim_{n\to\infty}\frac{1-|f(z_n)|}{1-|B(z_n)|}\in(0,\infty)\,,
	\end{equation}
	and there is a simply connected domain $V\subseteq\D$ with $1\in\partial V$ s.t.\
	\begin{equation}\label{eq:MBPJulia}
		\frac{|1-f(v)|^2}{1-|f(v)|^2}\leq A\frac{|1-B(v)|^2}{1-|B(v)|^2}\qquad\text{for all }v\in V.
	\end{equation}
	Moreover, if
	\begin{equation}\label{eq:MBPfequalsBnontangonsequenceintegratedNeharistrong}
		f(z_n)=B(z_n)+o(|1-z_n|)\quad \text{as }n\rightarrow\infty,
	\end{equation}
	then $A=1$.
\end{lemma}
\begin{remark}
	\begin{enumerate}[(a)]
		\item The case $B(z)=z$ in Lemma~\ref{lem:MBPBlaschkeJulia} recovers the classical Julia inequality (see e.g.~\cite[p.~63]{shapiroCompositionOperatorsClassical1993}) since in this case one can take $V=\D$.
		\item In general, one cannot take $V=\D$. For example, if $f(z)=z^3$, $B(z)=z^2$, the inequality~\eqref{eq:MBPJulia} fails for points near $-1$. Moreover, the proof of Lemma~\ref{lem:MBPBlaschkeJulia} does not uniquely determine the set~$V$. Given~$f$ and ~$B$ it is an interesting question to determine the largest possible set of points such that~\eqref{eq:MBPJulia} holds.
	\end{enumerate}
\end{remark}
In the proof of our main result, Theorem~\ref{thm:sequenceBurnsKrantzforMBP}, we make particular use of the following consequence of Lemma~\ref{lem:MBPBlaschkeJulia}.
\begin{corollary}\label{cor:MBPBlaschkeJulia}
	Under the assumptions and notations of Lemma \ref{lem:MBPBlaschkeJulia} the holomorphic function
	\begin{equation}\label{eq:MBPrealpartofharmonicfunction}
		\D\ni z\mapsto\frac{1+f(z)}{1-f(z)}-A\,\frac{1+B(z)}{1-B(z)}
	\end{equation}
	has non-negative real part on $V$. Moreover, if
	\begin{equation}\label{eq:MBPfequalsBstrongCorBlaschkeJulia}
		f(z_n)=B(z_n)+o(|1-z_n|^3)\quad \text{as }n\rightarrow\infty\,,
	\end{equation}
	then
	\begin{equation}\label{eq:holfunctionissmall}
		\frac{1+f(z_n)}{1-f(z_n)}-\frac{1+B(z_n)}{1-B(z_n)}=o(|1-z_n|)\quad \text{as }n\rightarrow\infty\,.
	\end{equation}
\end{corollary}
Recall the \emph{(two-point) Schwarz-Pick inequality} for holomorphic functions $f:\D\to\D$, namely
\begin{equation}\label{eq:SchwarzPick}
	\Bigg\vert \frac{f(z)-f(v)}{1-\cc{f(z)}f(v)}\Bigg\vert \leq \Bigg\vert\frac{z-v}{1-\cc{z}v}\Bigg\vert\qquad (z,v\in\D).
\end{equation}
Rewriting~\eqref{eq:SchwarzPick} and taking appropriate limits in order to use~\eqref{eq:JuliaCoefficient} gives one proof of the classical Julia inequality; the details can be found e.g.\ in~\cite[Lem.~1.6.2]{garciaFiniteBlaschkeProducts2018a}. Our idea for the proof of Lemma~\ref{lem:MBPBlaschkeJulia} is roughly the same. For this purpose, we need a ``Blaschke version'' of~\eqref{eq:SchwarzPick} first.
\begin{lemma}\label{lem:MBPintegratedNehari}
	Let $f:\D\to\D$ be a holomorphic function and $B$ a maximal Blaschke product for~$f$. Further, let $V\subseteq\D$ such that $B$ is injective on $V$ and such that $B(V)$ is hyperbolically convex. Then
	\begin{equation}\label{eq:MBPintegratedNehari}
		\left\vert\frac{f(v)-f(z)}{1-\cc{f(z)}f(v)}\right\vert\leq\left\vert\frac{B(v)-B(z)}{1-\cc{B(z)}B(v)}\right\vert\qquad\text{for all }z,v\in V\,.
	\end{equation}
\end{lemma}
\begin{proof}
	We make use of the following two (hyperbolic) geometric observations
	: Let $z,v\in\D$. First, in terms of the hyperbolic metric~$\dD$ (see~\eqref{eq:HyperbolicMetric}), inequality~\eqref{eq:MBPintegratedNehari} is equivalent to
	\begin{equation}
		\dD\big(f(z),f(v)\big)\leq \dD\big(B(z),B(v)\big).
	\end{equation}
	Second, if~$\gamma$ is any curve in~$\D$ connecting~$z$ and~$v$, then $f\circ\gamma$ is a curve connecting~$f(z)$ and~$f(v)$, and it follows
	\begin{equation}\label{eq:MBPNehariahypdistance1}
		\dD\big(f(z),f(v)\big)\leq \ell_h(f\circ \gamma) = \int_{\gamma}\frac{|f'(t)|}{1-|f(t)|^2} \, dt \leq \int_{\gamma}\frac{|B'(t)|}{1-|B(t)|^2} \, dt.
	\end{equation}
	In the last step we have used Theorem~\ref{th:NehariSchwarz}. Assume for a moment that we can find~$\gamma$ such that~$B$ is injective on the trace of~$\gamma$ and satisfies $B\circ\gamma=[B(z),B(v)]_h$. In this case, we can conclude that
	\begin{equation}\label{eq:MBPNehariahypdistance2}
		\int_{\gamma}\frac{|B'(t)|}{1-|B(t)|^2}dt=\int_{B\circ\gamma}\frac{1}{1-|t|^2}dt=\int_{[B(v),B(z)]_h}\frac{1}{1-|t|^2}dt=\dD\big(B(v),B(z)\big).
	\end{equation}
	Therefore, it remains to show that such a curve~$\gamma$ exists: By assumption, $\hat{B}:=B\vert_V:V\to B(V)$ is bijective. Therefore, for every $u,w\in B(V)$ we set $\gamma:=\hat{B}^{-1}([u,w]_h)$.
\end{proof}
Having Lemma~\ref{lem:Stolzendinclusions} in mind, note that the set~$V$ in Lemma~\ref{lem:MBPintegratedNehari} can be chosen as an appropriate end of a Stolz region. Thus, we are now in a position to give the proof of Lemma~\ref{lem:MBPBlaschkeJulia}.
	\begin{proof}[Proof of Lemma~\ref{lem:MBPBlaschkeJulia}]
		We apply Lemma~\ref{lem:MBPsequenceangularderivative} which implies that
		\begin{equation}
			A=\lim_{n\to\infty}\frac{1-|f(z_n)|}{1-|B(z_n)|}=\lim_{n\to\infty}\frac{1-|f(z_n)|}{1-|z_n|}\frac{1-|z_n|}{1-|B(z_n)|}=\frac{f'(1)}{B'(1)}\in(0,\infty).
		\end{equation}
		Moreover, if we additionally assume \eqref{eq:MBPfequalsBnontangonsequenceintegratedNeharistrong}, then  Lemma~\ref{lem:MBPsequenceangularderivative} shows $f'(1)=B'(1)$. Hence, $A=1$ in this case.

		Next, since $(z_n)$ converges non-tangentially to~1, we can fix $m>0$ such that $z_n\in S(6m/7,1)$ eventually. Choose $M>0$ resp.\ $V\subseteq\D$ according to Corollary~\ref{cor:MBPregionwhereBisinjective} for $\xi=1$ such that $B(V)=E(m,1,M)$ and $E(6m/7,1,M\alpha_B)\subseteq V$. Then we can apply Lemma \ref{lem:MBPintegratedNehari} for $V$ which gives
		\begin{equation}
			\left\vert\frac{f(v)-f(z)}{1-\cc{f(z)}f(v)}\right\vert\leq\left\vert\frac{B(v)-B(z)}{1-\cc{B(z)}B(v)}\right\vert\qquad\text{for all }z,v\in V.
		\end{equation}
		The previous inequality is equivalent to
		\begin{equation}\label{eq:auxillaryJuliabeforetakinglimits}
			\frac{|1-\cc{f(z)}f(v)|^2}{1-|f(v)|^2}\leq\frac{1-|f(z)|}{1-|B(z)|}\frac{|1-\cc{B(z)}B(v)|^2}{1-|B(v)|^2}\frac{1+|f(z)|}{1+|B(z)|} \qquad\text{for all }z,v\in V.
		\end{equation}
		 The assumption $z_n\to1$ allows us to find an index $N$ such that $z_n\in E(6m/7,1,M\alpha_B)$ for all $n\geq N$. Consequently, $B(z_n)\subseteq E(m,1,M)$ for all $n\geq N$. Therefore, we can choose $z=z_n$ for $n\geq N$ in~\eqref{eq:auxillaryJuliabeforetakinglimits}. Taking the limit $n\to\infty$ yields \eqref{eq:MBPJulia}.
	\end{proof}
	\begin{proof}[Proof of Corollary~\ref{cor:MBPBlaschkeJulia}]
		It immediately follows from~\eqref{eq:MBPJulia} that the real part of~\eqref{eq:MBPrealpartofharmonicfunction} is non-negative for all $z\in V$. For the additional statement, we adapt the argumentation in \cite[Prop.~3.2]{baraccoBurnsKrantzTypeTheorem2006}: \eqref{eq:MBPfequalsBstrongCorBlaschkeJulia} and $\alpha_B\in(0,\infty)$ guarantee that
		\begin{equation}
			\frac{f(z_n)-B(z_n)}{1-B(z_n)}=o(|1-z_n|^2).
		\end{equation}
		Therefore, given $\eps\in(0,1)$, for $n$ sufficiently large, we have
		\begin{equation}
			\left\vert\frac{f(z_n)-B(z_n)}{1-B(z_n)}\right\vert\leq\eps<1.
		\end{equation}
		For all of those $n$ we can compute
		\begin{align}
			\frac{1+f(z_n)}{1-f(z_n)}&=\frac{\big(1+f(z_n)\big)/\big(1-B(z_n)\big)}{1-\big(f(z_n)-B(z_n)\big)/\big(1-B(z_n)\big)}\\
			&=\frac{1+B(z_n)}{1-B(z_n)}+\frac{f(z_n)-B(z_n)}{1-B(z_n)}+\frac{1+f(z_n)}{1-B(z_n)}\sum_{k=1}^\infty\left(\frac{f(z_n)-B(z_n)}{1-B(z_n)}\right)^k\\
			&=\frac{1+B(z_n)}{1-B(z_n)}+o(|1-z_n|).\qedhere
		\end{align}
	\end{proof}

\section{Proof of Theorem~\ref{thm:sequenceBurnsKrantzforMBP}}\label{sec:proof:thm:sequenceBurnsKrantzforMBP}

	\begin{proof}[Proof of Theorem \ref{thm:sequenceBurnsKrantzforMBP}]
		Without loss of generality we assume $\xi=B(\xi)=1$. Our goal is to apply Theorem \ref{thm:MBPsequenceBurnsKrantzforhypderivatives}. For this purpose we need to understand the behaviour of the quotient
		\[\frac{|f'(z_n)|}{|B'(z_n)|}\frac{1-|B(z_n)|^2}{1-|f(z_n)|^2}\qquad \text{as }n\to\infty.\] By~\eqref{eq:MBPfequalsBnontangonsequence} and $\alpha_B\in(0,\infty)$ we have
		\begin{equation}\label{eq:MBPquotientofderivativesisotwo1}
			\frac{1-|f(z_n)|^2}{1-|B(z_n)|^2}=1+\frac{|B(z_n)|^2-|f(z_n)|^2}{1-|B(z_n)|^2}
			=1+\frac{o(|1-z_n|^3)}{1-|B(z_n)|^2}=1+o(|1-z_n|^2)
		\end{equation}
		as $n\to\infty$. It turns out that the quotient $|f'(z_n)/B'(z_n)|$ has the same asymptotic behaviour along~$(z_n)$. We divide the proof of this claim into several steps.

		\textsc{Step 1:} Fix $m>0$ such that $z_n\in S(4m/7,1)$ eventually. Choose $M>0$ resp.\ $V\subseteq\D$ according to Corollary~\ref{cor:MBPregionwhereBisinjective} such that $B(V)=E(m,1,M)$ and the corresponding inclusions hold. Denote by~$G$ a conformal map from~$\D$ onto~$V$. We first show that~$G$ extends to a homeomorphism of~$\cc{\D}$ onto~$\cc{V}$. In particular, this then allows us to assume $G(1)=1$.

		Indeed, Corollary~\ref{cor:MBPregionwhereBisinjective} shows that~$B$ is injective on~$\cc{V}\cap\D$. Combined with~$B$ having angular limit $B(1)=1$, this shows that the map
		\[\tilde{B}:\cc{V}\to\cc{B(V)}=\cc{E(m,1,M)},\quad \tilde{B}(v)=B(v)\]
		is continuous. Moreover, since $|B(v)|<1$ for all $v\in\cc{V}\cap\D$, the map $\tilde{B}$ is injective. Thus, as a continuous bijective map on a compact set, $\tilde{B}$ has a continuous inverse.\\
		Further observe that $E(m,\xi,M)$ is a Jordan domain (i.e. $\partial E(m,\xi,M)$ is a Jordan curve, that is an injective and continuous curve). Therefore, there is a conformal map~$\phi$ of~$\D$ onto~$E(m,1,M)$ that extends to a homeomorphism of~$\cc{\D}$ onto~$\cc{E(m,1,M)}$, too (see e.g.~\cite[Th.~2.6 and Cor.~2.8]{pommerenkeBoundaryBehaviourConformal1992}). If we assume $\phi(1)=1$, then we can choose~$G:=\tilde{B}^{-1}\circ\phi$.

		\textsc{Step 2:} By Corollary~\ref{cor:MBPBlaschkeJulia} the function
		\[\D\ni w\mapsto\frac{1+f(G(w))}{1-f(G(w))}-\frac{1+B(G(w))}{1-B(G(w))}\]
		has non-negative real part on $\D$. Hence, we find $F:V\to\D$ holomorphic such that
		\begin{equation}\label{eq:MBPuwithFandG}
			\frac{1+f(G(w))}{1-f(G(w))}-\frac{1+B(G(w))}{1-B(G(w))}=\frac{1+F(G(w))}{1-F(G(w))}\qquad\text{for all }w\in\D\,.
		\end{equation}
		Moreover, by \eqref{eq:holfunctionissmall} we have
		\begin{equation}\label{eq:MBPrealpartquotientFisotwo}
			\Re\left(\frac{1+F(z_n)}{1-F(z_n)}\right)=o(|1-z_n|)\quad\text{as }n\to\infty.
		\end{equation}
		For $w\in\D$ set $z:=G(w)$. Following an idea in~\cite[p.~9ff.]{ahlforsConformalInvariantsTopics2010} we differentiate~\eqref{eq:MBPuwithFandG} which leads to
		\begin{equation}\label{eq:MBPuwithFandGDifferentiated}
			\frac{f'(z)}{B'(z)}=\frac{(F\circ G)'(w)}{(1-(F\circ G)(w))^2}\frac{(1-f(z))^2}{G'(w)B'(z)}+\left(\frac{1-f(z)}{1-B(z)}\right)^2.
		\end{equation}

		\textsc{Step 2a:} For $z=z_n$ (resp.\ $w=G^{-1}(z_n)$) the assumption \eqref{eq:MBPfequalsBnontangonsequence} and $\alpha_B\in(0,\infty)$ imply
		\begin{equation}\label{eq:MBPquotientofderivativesisotwo2}
			\left(\frac{1-f(z_n)}{1-B(z_n)}\right)^2=\left(1+\frac{o(|1-z_n|^3)}{1-B(z_n)}\right)^2=1+o(|1-z_n|^2)\quad\text{as }n\to\infty.
		\end{equation}

		\textsc{Step 2b:} We prove that a similar estimate holds for the first term on the RHS in \eqref{eq:MBPuwithFandGDifferentiated}. More specifically, taking into account that $\alpha_B\in(0.\infty)$ implies $B'(z_n)=O(1)$ as $n\to\infty$, we show that
		\begin{equation}
			I(z_n):=\frac{(F\circ G)'(G^{-1}(z_n))}{(1-(F\circ G)(G^{-1}(z_n))^2}\frac{(1-f(z_n))^2}{G'(G^{-1}(z_n))}=o(|1-z_n|^2)\quad\text{as }n\to\infty.
		\end{equation}
		In fact, applying the Schwarz-Pick lemma to the holomorphic function $F\circ G:\D\to\D$ yields
		\begin{align}
			|I(z_n)|&\leq \frac{1}{1-|G^{-1}(z_n)|^2}\frac{1-|(F\circ G)(G^{-1}(z_n))|^2}{|1-(F\circ G)(G^{-1}(z_n))|^2}\frac{|1-f(z_n)|^2}{|G'(G^{-1}(z_n))|}\\
			&=\Re\left(\frac{1+F(z_n)}{1-F(z_n)}\right)\left\vert\frac{1-f(z_n)}{1-z_n}\right\vert^2\frac{|1-G^{-1}(z_n)|}{1-|G^{-1}(z_n)|}\frac{|1-z_n||(G^{-1})'(z_n)|}{|1-G^{-1}(z_n)|}\frac{|1-z_n|}{1+|G^{-1}(z_n)|}\,.
		\end{align}
		As $n\to\infty$, the first factor is $o(|1-z_n|)$ by \eqref{eq:MBPrealpartquotientFisotwo}. The second factor is~$O(1)$ by Lemma~\ref{lem:MBPsequenceangularderivative}. The last factor is $O(|1-z_n|)$. Therefore, if we can show that the third and fourth factor are both~$O(1)$, we could conclude
		\begin{equation}\label{eq:MBPquotientofderivativesisotwo3}
			I(z_n)=o(|1-z_n|^2)\quad\text{as }n\to\infty.
		\end{equation}

		\textsc{Step 2c:} Assume for a moment that~\eqref{eq:MBPquotientofderivativesisotwo3} holds. Together with~\eqref{eq:MBPquotientofderivativesisotwo1}, \eqref{eq:MBPuwithFandGDifferentiated} and~\eqref{eq:MBPquotientofderivativesisotwo2} this then implies
		\begin{equation}
			\frac{|f'(z_n)|}{|B'(z_n)|}\frac{1-|B(z_n)|^2}{1-|f(z_n)|^2}=1+o\big(|1-z_n|^2\big)\quad\text{as }n\to\infty\,.
		\end{equation}
		Hence, we can apply Theorem \ref{thm:MBPsequenceBurnsKrantzforhypderivatives} and conclude that $f=T\circ B$ for some $T\in\Aut(\D)$. Finally, \eqref{eq:MBPfequalsBnontangonsequence} implies that $T(z)=z$ on $\D$.

		\textsc{Step 3:} In view of \textsc{Step 2c} it remains to prove~\eqref{eq:MBPquotientofderivativesisotwo3} or, equivalently,
		\begin{equation}\label{eq:factorswithGinverseareO1}
			\frac{|1-G^{-1}(z_n)|}{1-|G^{-1}(z_n)|}=O(1)\quad\text{and}\quad \frac{|1-z_n||(G^{-1})'(z_n)|}{|1-G^{-1}(z_n)|}=O(1)\qquad\text{as }n\to\infty.
		\end{equation}

		\textsc{Step 3a:} We need some preliminary observations: Choose $\beta\in[0,\pi/2]$ such that $\tan(\beta/2)=\tanh(m)$. Then the map $C:\D\to\RH:=\{w\in\C \, :\, \Re w>0\}$ defined by $C(z)=(1-z)/(1+z)$ maps $S(m,1)$ onto the sector
		\[S_{\RH}(\beta,0):=\{w=r e^{i\theta}\in\RH\,:\, r>0, \,-\beta < \theta<\beta\}\]
		(see~\cite[Prop.~2.2.7]{abateHolomorphicDynamicsHyperbolic2022}). The map $\rho_\beta:S_{\RH}(0,\beta)\to\RH$, $w\mapsto w^{\pi/(2\beta)}$ is well-defined and, in particular, onto. Therefore, $\psi=C^{-1}\circ \rho_\beta\circ C$ maps $S(m,1)$ onto $\D$. If we are given $m'<m$, then $\psi(S(m',1))=S(\tilde{m},1)\subsetneq\D$ for some $\tilde{m}>0$.

		Now consider the open set $H:=\psi(E(m,1,M))\subseteq\D$ and a conformal map $\varphi$ mapping~$\D$ onto~$H$ with $\varphi(1)=1$. By construction, there is a subarc $J\subseteq\partial\D$ such that $1\in J$ and ${\partial H\cap \partial \D=J}$. Therefore, we find an open subarc $I\subseteq\partial \D$ containing 1 such that $\varphi(I)\subseteq J$. The Schwarz reflection principle (see e.g.~\cite[p.~4]{pommerenkeBoundaryBehaviourConformal1992}) shows that $\varphi$ admits an analytic continuation to some neighborhood of 1. In particular, $\alpha_\varphi\in(0,\infty)$.

		\textsc{Step 3b:} In view of the map $G^{-1}$ and the domain $V$ fixed in \textsc{Step 1} of our proof, we can write $G^{-1}(v)=(\varphi^{-1}\circ \psi \circ B)(v)$ for all $v\in V$. Our choice of $m$ and $M$ combined with Corollary~\ref{cor:MBPregionwhereBisinjective} guarantees that $(B(z_n))_{n\geq N}\subseteq E(5m/7,1,2M)\subsetneq E(m,1,M)$ for some index~$N$. Therefore, we find $m'>0$ such that $(\psi(B(z_n)))_{n\geq N}\subseteq S(m',1)\cap H$. Moreover, since $\alpha_\varphi\in(0,\infty)$, we can apply Lemma~\ref{lem:Stolzendinclusions} (for $\eps=\min\{1/3m',\log2\}$) and conclude that there is $M'>0$ such that
		\[E(m',1,M')\subseteq \varphi\big(E(4m'/3,1,\alpha_\varphi M'/2)\big)	.\]
		Consequently, since $(\psi(B(z_n)))_{n\geq N'}\subseteq E(m',1,M')$ for some index $N'\geq N$, we conclude that
		\[\big(G^{-1}(z_n)\big)_{n\geq N'}=\big(\varphi^{-1}(\psi(B(z_n)))\big)_{n\geq N'}\subseteq  E(4m'/3,1,\alpha_\varphi M'/2).\]
		In other words, $(G^{-1}(z_n))$ converges non-tangentially to~1 which shows that
		\[\frac{|1-G^{-1}(z_n)|}{1-|G^{-1}(z_n)|}=O(1)\qquad\text{as }n\to\infty.\]

		\textsc{Step 3c:} In order to prove the second identity in~\eqref{eq:factorswithGinverseareO1}, we write
		\begin{multline}\label{eq:VisserOstrowskiforGinverse}
			\frac{(1-z_n)(G^{-1})'(z_n)}{1-G^{-1}(z_n)}=\frac{(1-z_n)(\varphi^{-1}\circ\psi\circ B)'(z_n)}{1-(\varphi^{-1}\circ\psi\circ B)(z_n)}
			\\
			=\frac{(1-z_n)B'(z_n)}{1-B(z_n)}
			\frac{\big(1-B(z_n)\big)\psi'\big(B(z_n)\big)}{1-\psi(B(z_n))}
			\frac{\big(1-(\psi\circ B)(z_n)\big)(\varphi^{-1})'\big(\psi( B(z_n))\big)}{1-\varphi^{-1}\big((\psi\circ B)(z_n)\big)}.
		\end{multline}
		Since $z_n\to1$ non-tangentially and $\alpha_B\in(0,\infty)$, the first factor is~$O(1)$. Similarly, since~$\varphi^{-1}$ is holomorphic at 1 and $\varphi^{-1}(1)=1$, the third factor is~$O(1)$, too. It remains to consider the second factor in~\eqref{eq:VisserOstrowskiforGinverse}. Denote $u_n:=B(z_n)$. Then
		\begin{equation}
			\frac{(1-u_n)\psi'(u_n)}{1-\psi(u_n)}
			=\frac{1+u_n}{1+\psi(u_n)}\frac{1-u_n}{1+u_n}\frac{1+\psi(u_n)}{1-\psi(u_n)}\psi'(u_n)=\frac{1+u_n}{1+\psi(u_n)}\frac{C(u_n)}{C(\psi(u_n))}\psi'(u_n).
		\end{equation}
		Using the explicit form of $\psi'$, the property $C=C^{-1}$ and $C\circ\psi=\rho_\beta\circ C$, we get
		\begin{align*}
			\frac{(1-u_n)\psi'(u_n)}{1-\psi(u_n)}
			&=\frac{\pi}{2\beta}\frac{1+u_n}{1+\psi(u_n)}\frac{4}{\big(1+\rho_\beta(C(u_n))\big)^2(1+u_n)^2}.
		\end{align*}
		Now $\psi(1)=1$ and $C(1)=0=\rho_\beta(0)$ as well as $u_n\to1$ as $n\to\infty$ imply
		\begin{equation}
			\lim_{n\to\infty}\frac{(1-u_n)\psi'(u_n)}{1-\psi(u_n)}=\frac{\pi}{2\beta}.
		\end{equation}
		This concludes the proof of
		\[\frac{|1-z_n||(G^{-1})'(z_n)|}{|1-G^{-1}(z_n)|}=O(1)\qquad\text{as }n\to\infty\]
		and hence the proof of Theorem~\ref{thm:sequenceBurnsKrantzforMBP}.
	\end{proof}

\begin{concludingremarks}
		\begin{enumerate}[(a)]
			\item\label{it:MBPBurnsKrantzNonTangential} If we replace the assumption \eqref{eq:MBPfequalsBnontangonsequence} in Theorem~\ref{thm:sequenceBurnsKrantzforMBP} by
			\begin{equation}
				f(z)=B(z)+o(|\xi-z|^3)\quad \text{as }z\rightarrow\xi\text{ non-tangentially},
			\end{equation}
			then one can give a different shorter proof using Cauchy's integral formula. More precisely, one can adapt the argumentation in \cite[Prop.~8.1]{bracciNewSchwarzPickLemma2023}, see also \cite[Th.~2.7.4]{abateHolomorphicDynamicsHyperbolic2022}.
			\item The proof of Theorem~\ref{thm:sequenceBurnsKrantzforMBP} can also be simplified if we assume that the MBP~$B$ is a finite Blaschke product or, more general, that~$B$ is holomorphic at~$1$: in this case one can construct the map~$G$ (from the proof of Theorem~\ref{thm:sequenceBurnsKrantzforMBP}) such that~$G$ and $G^{-1}$ are holomorphic at~$1$, too. Then, the claim in \textsc{Step 2} is obvious.
			\item The exponent 3 in Theorem~\ref{thm:sequenceBurnsKrantzforMBP} is sharp if we assume \eqref{eq:MBPfequalsBnontangonsequence} non-tangentially (see Part~\eqref{it:MBPBurnsKrantzNonTangential} above): Define $g:\D\to\D$, $g(z)=(1+3z^2)/(3+z^2)$ and let~$B$ be a MBP (for an arbitrary holomorphic self-map of~$\D$) with $\alpha_B\in(0,\infty)$. Then set $f:=g\circ B$.
		\end{enumerate}
\end{concludingremarks}

\section*{Acknowledgements}

The author thanks Oliver Roth for countless helpful and inspiring discussions.

\end{document}